\documentclass[10pt]{amsart}
\usepackage{amssymb}
\usepackage{amscd}
\usepackage{comment}

\theoremstyle{plain}
\newtheorem{theorem}{Theorem}

\theoremstyle{definition}
\newtheorem{question}{Question}

\newtheorem{proposition}{Proposition}[section]

\newtheorem{corollary}[proposition]{Corollary}

\newtheorem{lemma}[proposition]{Lemma}
\newtheorem{definition}[proposition]{Definition}
\theoremstyle{remark}

\DeclareMathOperator{\dom}{dom}
\DeclareMathOperator{\cl}{cl}

\DeclareMathOperator{\ran}{ran}

\newcommand{\sk}{\vskip.05in}

\newcommand{\restr}{\upharpoonright}
\newcommand{\forces}{\Vdash}

\newcommand{\f}{\mathcal{F}}

\newcommand{\subs}{\subseteq}
\newcommand{\sups}{\supseteq}

\numberwithin{equation}{section}
\begin{document}
\title{{C}{H} and the Moore-Mrowka Problem}
\author{Alan Dow}
\address{Department of Mathematics\\
University of North Carolina, Charlotte\\
Charlotte, NC 28223}
\email{adow@uncc.edu}
\author{Todd Eisworth}
\address{Department of Mathematics\\
         Ohio University\\
         Athens, OH 45701}
\email{eisworth@math.ohiou.edu}

\date{\today}
\begin{abstract}
We show that the Continuum Hypothesis is consistent with all regular spaces of hereditarily countable $\pi$-character being $C$-closed. This gives us
a model of {\sf ZFC} in which the Continuum Hypothesis holds and compact Hausdorff spaces of countable tightness are sequential.
\end{abstract}

\maketitle

\section{Introduction}

Our goal in this paper is to prove that the Continuum Hypothesis is consistent with the following statement:

\begin{list}{$\circledast$}{\setlength{\leftmargin}{.5in}\setlength{\rightmargin}{.5in}}
\item Regular spaces that are hereditarily of countable $\pi$--character
are $C$-closed.
\end{list}

The principle $\circledast$ may seem technical (and mysterious given the lack of definitions), but it turns out to be of interest for a couple of reasons:
\begin{itemize}
\item It is unknown if $\circledast$ is a consequence of the Proper Forcing Axiom or even Martin's Maximum, and our consistency proof seems to need that
the Continuum Hypothesis holds in the final model.
\sk
\item More importantly, in the presence of $2^{\aleph_0}<2^{\aleph_1}$, $\circledast$ is strong enough to imply that compact spaces of countable tightness are sequential, and so our model provides the final piece to the
solution of the well-known ``Moore-Mrowka + {\sf CH} problem'' in set-theoretic topology.
\sk
\end{itemize}

The Moore-Mrowka problem was long considered to be one of the major problems
in set-theoretic topology. The reader may be interested in a brief review.
The problem, asking if every countably tight
 compact space was sequential,
was raised by Moore and Mrowka in 1964. In the
mid-seventies two of the most famous examples in topology,
Ostaszewski's S-space \cite{Ost1976} and Fedorchuk's S-space \cite{Fed1976},
established that $\diamondsuit$ implied there were counterexamples.
The Moore-Mrowka statement was proven to be consistent by
Balogh \cite{Bal1989} by showing that it was a consequence of the
 proper forcing axiom. Nevertheless, since $\diamondsuit$ implies {\sf CH}, Arhangelskii asked
\cite{Arh1978} (Problem 26) if the Continuum was sufficient to produce a
counterexample. This problem was again raised by Shakhmatov
\cite{Shak1992}(2.13) in the influential Recent Progress in General Topology.

Before moving on, we provide the reader with some definitions, looking first at some important cardinal functions for topological spaces.

\begin{definition}
Let $X$ be a topological space.
\begin{enumerate}
\item $z\in X$ is a point of countable tightness in $X$ ($t(z, X)=\aleph_0)$ if whenever $A\subseteq X$ and $z\in\cl_X(A)$, there is a countable
$A_0\subseteq A$ such that $z\in\cl_X(A_0)$.
\sk
\item $X$ is countably tight ($t(X)=\aleph_0$) if $t(z, X)=\aleph_0$ for every $z\in X$.
\sk
\item $X$ has countable $\pi$-character ($\pi\chi(X)=\aleph_0$) if   for any point $x\in X$ there is a countable collection $\{U_n:n\in\omega\}$ of
non-empty open
sets such that for any open neighborhood $U$ of $x$ there is an $n$ with $U_n\subseteq U$.  Note that we are not requiring that $x$ is a member of $U_n$, so
this is a weakening of first countability.
\sk
\item $X$ is hereditarily of countable $\pi$-character ($h\pi\chi(X)=\aleph_0$) if $\pi\chi(Y)=\aleph_0$ whenever $Y\subseteq X$.
\sk
\end{enumerate}
\end{definition}

An elementary argument shows that $h\pi\chi(X)=\aleph_0\Longrightarrow t(X)=\aleph_0$,
and the reverse implication is true if $X$ is compact (Hausdorff) by a deep theorem of Sapirovskii~\cite{sapirovskii}.

\begin{definition}
Let $X$ be a topological space.
\begin{enumerate}
\item $X$ is countably compact if every infinite subset has a point of accumulation, or equivalently, if every countable open cover of $X$ has a finite subcover.
\sk
\item $X$ is $C$-closed if every countably compact subset of $X$ is closed in $X$.
\sk
\item A subset $Y$ of $X$ is sequentially closed if whenever $\langle x_n:n<\omega\rangle$ is a convergent sequence of points from $Y$, the limit of the sequence is also in $Y$,
that is, $Y$ is closed under the operation of taking limits of convergent sequences.
\sk
\item $X$ is sequential if every sequentially closed subset of $X$ is closed.
\sk
\end{enumerate}
\end{definition}

Given the above definitions, we can now illustrate the power of our principle $\circledast$:

\begin{proposition}\label{chmm}
\leavevmode
\begin{enumerate}
\item Assume $2^{\aleph_0}<2^{\aleph_1} + \circledast$.  Then  compact spaces of countable tightness are sequential.
\sk
\item  $\circledast$ implies that if $X$ is a countably compact regular space satisfying $h\pi\chi(X)=\aleph_0$, then every non-isolated point in $X$ is the limit of a non-trivial convergent sequence.
\end{enumerate}
\end{proposition}
\begin{proof}
First, suppose $X$ is compact and of countable tightness.  By Sapirovskii's Theorem \cite{sapirovskii}, we know $h\pi\chi(X)=\aleph_0$, and so by $\circledast$ we know
that $X$ is $C$-closed.  By a result of Ismail and Nyikos \cite{ismailnyikos}, $2^{\aleph_0}<2^{\aleph_1}$ implies that a compact Hausdorff space is $C$-closed if and only if it sequential, and we are done.

For the second, suppose $X$ is as described, and $x$ is a non-isolated point.  Since $X$ is $C$-closed, we know that $X\setminus\{x\}$ is not countably compact.  This means we can find a (countably infinite) set
\begin{equation*}
\{x_n:n\in\omega\}\subseteq X\setminus\{x\}
\end{equation*}
that is closed and discrete in $X\setminus\{x\}$.  Since $X$ is countably compact, it follows that $\{x_n:n\in\omega\}$ must
converge to $x$ in $X$.
\end{proof}

By (1) above, any model of {\sf CH} + $\circledast$ gives us a model of {\sf CH} in which compact spaces of countable tightness are sequential.  At this point, we do not know if $\circledast$ can hold in a model where {\sf CH} fails.

The conclusion of (2) says that $X$ is {\em $C$-sequential}, a notion introduced and studied by Ran{\v{c}}in~\cite{rancin}.  Hajnal and Juh{\'a}sz show that {\sf CH} implies the existence of a countably compact regular space of countable tightness with no non-trivial convergent sequences at all, so the use of hereditary $\pi$-character rather than tightness is critical.

\section{Bad triples}

\begin{definition}
We call $\vec{X}=\langle X, Y, z\rangle$ a {\em relevant triple} if
\begin{enumerate}
\item $X$ and $Y$ are regular separable topological spaces
\sk
\item the underlying set of $Y$ is $\omega_1$
\sk
\item $Y= X\cup\{z\}$
\sk
\item $X$ is countably compact, and
\sk
\item $z$ is not isolated in $Y$.
\sk
\end{enumerate}
A relevant triple $\langle X, Y, z\rangle$ is {\em bad}  (and called a {\em bad triple}) if in addition
\begin{enumerate}
\setcounter{enumi}{4}
\sk
\item $h\pi\chi(X)=\aleph_0$, and
\sk
\item $t(z, Y)=\aleph_0$.
\end{enumerate}
\end{definition}

Note that the separability of $Y$ follows from the separability of $X$, so this assumption is superfluous. Also, the requirement that the underlying
set of $Y$ is $\omega_1$ is included for technical reasons: if we require only that $|Y|=\aleph_1$, then clearly $Y$ is homeomorphic to some space
with underlying set $\omega_1$.

Bad triples are relevant to our project because of the following fact:

\begin{proposition}
\label{prop2}
Assume {\sf CH} holds.  If $\circledast$ fails, then there is a bad triple.
\end{proposition}
\begin{proof}
Assume $Z$ is a counterexample to $\circledast$, so $Z$ is regular, hereditarily of countable $\pi$-character, but not $C$-closed.  Let $W$
be a countably compact non-closed subset of $Z$,and choose $z\in\cl_Z(W)\setminus W$.

Since $h\pi\chi(Z)=\aleph_0$, we know $Z$ has countable tightness, and so there is a countable $W_0\subseteq W$ with $z\in\cl_Z(W_0)$.
Since $\cl_Z(W_0)$ is also countably compact and not closed, we can without loss of generality assume that $W$ is separable.

A separable regular space has weight at most $2^{\aleph_0}$. Since we are assuming {\sf CH} that means $w(W)$ is at most $\aleph_1$.
Certainly $w(W)$ is uncountable, as regular spaces with countable bases are metrizable and countably compact metrizable spaces are compact.
Thus,  $W$ has weight exactly $\aleph_1$.

Since the Continuum Hypothesis holds, we can find $M$ such that
\begin{itemize}
\item $M$ is an elementary submodel of $H(\chi)$ for some sufficiently large regular $\chi$,
\sk
\item $M$ has cardinality $\aleph_1$,
\sk
\item $M$ is closed under $\omega$-sequences, and
\sk
\item $Z$, $W$, and $z$ are all elements of $M$.
\end{itemize}

Now let $X= M\cap W$ (topologized as a subspace of $W$) and $Y= X\cup\{z\}$.  We claim that $\langle X, Y, z\rangle$ is a bad triple.

Elementarity arguments tell us that $X$ and $Y$ are regular, $X$ is hereditarily of countable $\pi$-character, $z$ is not isolated in $Y$, and
$t(z, Y)=\aleph_0$.  Note that $X$ (and hence $Y$) is separable, as $M$ will contain every member of a (countable) dense subset of $W$.

The only remaining issue of substance is whether $X$ is countably compact, but this follows easily using the fact that $M$ is closed under $\omega$-sequences:
any countably infinite subset $A$ of $X$ is a member of $M$, and by elementary $M$ will contain a limit point of $A$ (as $M$ knows
that $W$ is countably compact) and therefore $A$ has a limit point in $X$.
\end{proof}

\section{Adjoining a filter}

This section is critical as it contains the main new idea necessary to produce our desired model.

\begin{definition}
Suppose $\vec{X}=\langle X, Y, z\rangle$ is a bad triple.  We define
\begin{equation}
\mathbb{P}[\vec{X}]=\{A\subseteq X: \text{ $A$ is a separable closed subset of $X$ with } z\in\cl_Y(A)\}
\end{equation}
and order $\mathbb{P}[\vec{X}]$ by setting
\begin{equation}
A\leq B\Longleftrightarrow A\subseteq B.
\end{equation}
\end{definition}
 if $A$ is a closed subset of $X$ and $z\in\cl_Y(A)$, then there is a $B\in\mathbb{P}[\vec{X}]$ such that $B\subseteq A$.

\begin{proposition}
With $\vec{X}=\langle X, Y, z\rangle$ and $\mathbb{P}=\mathbb{P}_{\vec{X}}$ as above, we have:
\begin{enumerate}
\item If $A$ is a closed subset of $X$ with $z\in\cl_Y(A)$, then there is a $B\in\mathbb{P}[\vec{X}]$ with $B\subseteq A$.
\sk
\item $\mathbb{P}$ is a countably closed notion of forcing, and
\sk
\item for each open neighborhood $U$ of $z$ in $Y$, $D^U:=\{A\in \mathbb{P}: A\subseteq U\}$ is dense in $\mathbb{P}$.
\end{enumerate}
\end{proposition}
\begin{proof}
The first statement is trivial as $t(z, Y)=\aleph_0$.  For the second,  suppose $\langle A_n:n<\omega\rangle$ is a decreasing sequence of conditions
 in $\mathbb{P}$, and let  $A:=\bigcap\{A_n:n<\omega\}$. The set $A$ is clearly a closed subset of $X$, so in light of (1) it suffices to show $z\in\cl_Y(A)$.

Suppose this fails, and let $U$ be an open neighborhood of $z$ in $Y$ for which
\begin{equation}
\label{eqn3}
A\cap U=\emptyset.
\end{equation}
Since $Y$ is regular, we can find an open neighborhood $V$ of $z$ with $\cl_Y(V)\subseteq U$.

Now let $W=\cl_Y(V)\cap X$.  Since $W$ is a closed subset of $X$ it is
 countably compact.  Furthermore, for each $n$ we know $A_n\cap W\neq\emptyset$
because $A_n$ meets $V$. Thus, $\langle A_n\cap V:n<\omega\rangle$ is a decreasing sequence of non-empty closed subsets of the countably compact
space $W$. We conclude that
\begin{equation}
\bigcap\{A_n\cap V:n<\omega\}=A\cap V\neq\emptyset,
\end{equation}
which contradicts (\ref{eqn3}).

For (3), given $U$ and a condition $A\in\mathbb{P}$, we can choose an open neighborhood $V$ of $z$ with $\cl_Y(V)\subseteq U$, and the
set  $A\cap\cl_Y(V)$ contains a condition in $\mathbb{P}[\vec{X}]$ by (1).
\end{proof}

We will be using forcings like $\mathbb{P}[\vec{X}]$ in models where {\sf CH} holds, and the next proposition shows why this cardinal arithmetic
assumption is useful:

\begin{proposition}
\label{usefulprop}
Assume the Continuum Hypothesis, and suppose $\vec{X}=\langle X, Y, z\rangle$ and $\mathbb{P}=\mathbb{P}[\vec{X}]$ are as above.
Let $G$ be a generic subset of $\mathbb{P}$.   In the extension $V[G]$, if we define
\begin{equation}
\mathcal{F}:=\{A\subseteq X:  \text{$A$ closed and $A\supseteq B$ for some $B\in G$ }\}
\end{equation}
then
\begin{enumerate}
\item $\mathcal{F}$ is a maximal free filter of closed subsets of $X$, and
\sk
\item for any open neighborhood $U$ of $z$ in $Y$, there is an $A\in\mathcal{F}$ with $A\subseteq U$.
\end{enumerate}
\end{proposition}
\begin{proof}

As far as (1) goes, the only issue is whether $\mathcal{F}$ is maximal in $V[G]$.  It suffices to prove that whenever we have a condition
 $p\in\mathbb{P}$ and a $\mathbb{P}$-name $\dot A$ for which
\begin{equation}
\label{eqn4}
p\forces\dot A\text{ is a closed subset of $X$ that meets every member of $\dot G_{\mathbb{P}}$},
\end{equation}
we can find an extension $q$ of $p$ in $\mathbb{P}$ such that
\begin{equation}
q\forces q\subseteq \dot A.
\end{equation}

Since we have assumed {\sf CH}, we know that $Y$ has weight $\aleph_1$. The point $z$ cannot have a countable neighborhood base in $Y$ (as otherwise
$X$ would not be countably compact), so it follows that $z$ has character $\aleph_1$ in $Y$. Let $\{U_\alpha:\alpha<\omega_1\}$ enumerate
 a neighborhood base for $z$ in $Y$. Note that (\ref{eqn4}) shows us that
 $r\forces\dot A\cap r\neq\emptyset$ whenever $r\leq p$ in $\mathbb{P}$, and so we can construct a sequence $\langle p_\alpha:\alpha<\omega_1\rangle$ of
 conditions in $\mathbb{P}$ and sequence
 $\langle x_\alpha:\alpha<\omega_1\rangle$ of points in $X$ such that
 \begin{itemize}
 \item $p_0=p$
 \sk
 \item $p_\alpha=\bigcap\{p_\beta:\beta<\alpha\}$ for $\alpha$ limit
 \sk
 \item $p_{\alpha+1}\leq p_\alpha$
 \sk
 \item $p_{\alpha+1}\forces x_\alpha\in p_\alpha\cap U_\alpha\cap\dot A$.
 \sk
 \end{itemize}

 Since $t(z, Y)=\aleph_0$, we can choose $\alpha<\omega_1$ least for which
 \begin{equation}
 z\in\cl_Y\{x_\beta:\beta<\alpha\}.
 \end{equation}
 Note that this implies $\alpha$ is a limit ordinal, and
 \begin{equation}
 \beta<\alpha\Longrightarrow z\in\cl_Y\{x_\gamma:\beta<\gamma<\alpha\}.
 \end{equation}

Now define
 \begin{equation}
 r:=\cl_X\{x_\beta:\beta<\alpha\},
 \end{equation}
 clearly $r$ is a condition in $\mathbb{P}$. Moreover, given $\beta<\alpha$ we know
 \begin{equation}
 \{x_\gamma:\beta<\gamma<\alpha\}\subseteq r\cap p_\beta,
 \end{equation}
 and hence $r\cap p_\beta$ is a condition in $\mathbb{P}$ for each $\beta<\alpha$.

 Let us define
 \begin{equation}
 q:=\bigcap_{\beta<\alpha}r\cap p_\beta.
 \end{equation}
 We know that $q\in\mathbb{P}$ (see the proof that $\mathbb{P}$ is countably closed) and $q$ extends $p=p_0$ in $\mathbb{P}$. We have $q\leq p_{\beta+1}$ for each
 $\beta<\alpha$, and this means
 \begin{equation}
 q\forces\{x_\beta:\beta<\alpha\}\subseteq\dot A.
 \end{equation}
 By definition of $r$, we have
 \begin{equation}
 q\forces r\subseteq\dot A
 \end{equation}
 and since $q\subseteq r$, we conclude
 \begin{equation}
 q\forces q\subseteq \dot A
 \end{equation}
 as required.
\end{proof}

The following theorem summarizes our efforts in this section:

\begin{theorem}
\label{summary1}
Suppose {\sf CH} holds and $\vec{X} = \langle X, Y, z\rangle$ is a bad triple. Then
\begin{enumerate}
\item  the notion of forcing $\mathbb{P}[\vec{X}]$ is countably closed and of cardinality $\aleph_1$, and
\sk
\item if $G$ is a generic subset of $\mathbb{P}[\vec{X}]$, then in $V[G]$ there is an object $\mathcal{F}$ such that
\begin{enumerate}
\sk
\item $\mathcal{F}$ is a maximal free filter of of closed subsets of $X$,
\sk
\item $\mathcal{F}$ has a base of separable sets, and
\sk
\item for any open neighborhood $U$ of $z$ in $Y$, there is an $A\in\mathcal{F}$ with $A\subseteq U$.
\sk
\end{enumerate}
\end{enumerate}
\end{theorem}

\section{Destroying a counterexample}

In this section we examine the problem of ``killing'' a given bad triple. We show that bad triples can be destroyed by a reasonable notion of
forcing, as long as we have the Continuum Hypothesis available.

Our plan is as follows:  given a bad triple $\vec{X}=\langle X, Y, z\rangle$, we will first force with $\mathbb{P}[\vec{X}]$, and ask if $X$ is still hereditarily of countable $\pi$-character.  If the
answer is ``no'', then $\vec{X}$ is no longer a bad triple and we are done.  If the answer is ``yes'', then we
will do an additional forcing that achieves $t(z, Y)>\aleph_0$, and thus destroy the badness of $\vec{X}$ via a different route.

Most of this section will concentrate on the second notion of forcing mentioned above. This forcing was
isolated in previous work~\cite{EI17}, and we will review briefly its definition and salient properties.

\begin{definition}
\label{41}
We say that  a pair $\langle X, \mathcal{F}\rangle$ {\em satisfies the requirements of~\cite{EI17}} if
\begin{itemize}
\item $X$ is a countably compact, non-compact regular space satisfying $h\pi\chi(X)=\aleph_0$, and
\sk
\item $\mathcal{F}$ is a maximal free filter of closed subsets of $X$ with a separable base.
\sk
\end{itemize}
\end{definition}

We will use such $X$ and $\mathcal{F}$ to define the notion of forcing, but before we can do this we need
some more definitions lifted from earlier work:

\begin{definition}
Suppose $X$ and $\mathcal{F}$ are as above
\begin{enumerate}
\item A subset $A$ of $X$ is {\em large} if $A$ meets every set in $\mathcal{F}$. A set that is not large is called {\em small}.
\sk
\item A {\em promise} is a function whose domain is a large subset of $X$ such that for each $x\in\dom(f)$, $f(x)$ is an open neighborhood of $x$.
\sk
\end{enumerate}
\end{definition}

Given these definitions, we define the notion of forcing $\mathbb{P}[X,\mathcal{F}]$ as follows:

\begin{definition}
A forcing condition $p$ is a triple $(\sigma_p, A_p,\Phi_p)$ where
\begin{enumerate}
\item $\sigma_p$ is a one--to--one function from some countable ordinal into $X$, and we define
 $[p]:=\ran(\sigma_p)$
\item $A_p\in\f$
\item $\Phi_p$ is a countable set of promises (see below)
\item $\cl_X([p])\cap A_p=\emptyset$
\end{enumerate}
and a condition $q$ extends $p$ (written $q\leq p$) if
\begin{enumerate}
\setcounter{enumi}{4}
\item $\sigma_q\sups\sigma_p$
\item $A_q\subs A_p$
\item $\Phi_q\sups \Phi_p$
\item $[q]\setminus[p]\subs A_p$
\item if $f\in\Phi_p$, then the set
\begin{equation*}
Y(f, q, p):=\{x\in\dom f: [q]\setminus[p]\subs f(x)\}
\end{equation*}
is large, and $f\restr Y(f, q, p)\in\Phi_q$.
\end{enumerate}
\end{definition}

Most of the paper~\cite{EI17} is concerned with working out properties of the above notion of forcing. For example, the notion of forcing
is proper, and in the generic extension every countable sequence of ground model elements is already in the ground model.  We will be more systematic
describing the properties of $\mathbb{P}[X,\mathcal{F}]$ later, but for our immediate purposes we need the fact that it
adjoins a sequence $\langle x_\alpha:\alpha<\omega_1\rangle$ satisfying the following two conditions:
\begin{equation}
\cl_X(\{x_\beta:\beta<\alpha\})\notin\mathcal{F}\text{ for every $\alpha<\omega_1$},
\end{equation}
and
\begin{equation}
\text{for every }A\in \mathcal{F}\text{ there is an $\alpha$ such that }\{x_\beta:\alpha\leq\beta<\omega_1\}\subseteq A.
\end{equation}
The sequence just described is just $\bigcup\{\sigma_p:p\in G\}$, where $G$ is a generic subset of~$\mathbb{P}[X,\mathcal{F}]$.

The following lemma gives an additional property of $\mathbb{P}[X,\mathcal{F}]$ that will be needed for our argument:

\begin{lemma}
Suppose $X$ and $\mathcal{F}$ are as above, and $\mathcal{U}$ is an open cover of $X$. If $G$ is a generic subset of $\mathbb{P}[X,\mathcal{F}]$
then in $V[G]$  there is an $\alpha<\omega_1$ such that any
countable subset of $\{x_\beta:\alpha\leq\beta<\omega_1\}$ is covered by an element of $\mathcal{U}$
\end{lemma}
\begin{proof}
Suppose $\mathcal{U}$ is an open cover of $X$; our plan is to show that the set of conditions forcing the existence of such an $\alpha$ is dense. To that
end, let $p = (\sigma_p, A_p, \Phi_p)$ be a condition.  For each $x\in X$, choose $U_x\in\mathcal{U}$ with $x\in U_x$, and let
$f:X\rightarrow\mathcal{U}$ be the function mapping $x$ to $U_x$.  Note that $f$ is a promise, and
\begin{equation*}
q:=(\sigma_p, A_p, \Phi_p\cup\{f\})
\end{equation*}
is an extension of $p$ in $\mathbb{P}[X,\mathcal{F}]$.

Suppose now that $G$ is a generic subset of $\mathbb{P}$ containing $q$. We will work in the extension $V[G]$.
First, we set $\alpha = \dom(p)$. It suffices to show that whenever $\alpha<\beta<\omega_1$, the set $\{x_\gamma:\alpha\leq\gamma<\beta\}$
is covered by an element of $\mathcal{U}$.

Given such a $\beta$, we fix $r\in G$ with $\beta\subseteq\dom(r)$ (such an $r$ can be found by our discussion preceding the lemma).
Since $G$ is a filter on $\mathbb{P}[X,\mathcal{F}]$, we may assume that $r$ extends $q$. Since $f\in\Phi_q$, it follows that
\begin{equation*}
[r]\setminus[q]\subseteq f(x)\text{ for a large set of }x.
\end{equation*}
But $f(x)$ is always an element of $\mathcal{U}$, and
\begin{equation}
\{x_\gamma:\alpha\leq\gamma<\beta\}\subseteq [r]\setminus [p]=[r]\setminus [q],
\end{equation}
so we have what we need.
\end{proof}

Armed with the above, we show how $\mathbb{P}[X,\mathcal{F}]$ can be used to destroy certain bad triples:

\begin{lemma}
\label{usefullemma}
Suppose $\bar{X}=\langle X, Y, z\rangle$ is a bad triple. Further assume there is an $\mathcal{F}$ satisfying:
\begin{itemize}
\item $\mathcal{F}$ is a maximal free filter of closed subsets of $X$,
\sk
\item for any open neighborhood $U$ of $z$ in $Y$, there is an $A\in\mathcal{F}$ with $A\subseteq U$, and
\sk
\item $\mathcal{F}$ has a base of separable sets.
\sk
\end{itemize}
Let $G$ be a generic subset of $\mathbb{P}[X,\mathcal{F}]$.  Then
\begin{equation}
V[G]\models \text{``}t(z, Y)>\aleph_0\text{"}.
\end{equation}
\end{lemma}
\begin{proof}
We have already mentioned that in the generic extension there is a sequence $\langle x_\alpha:\alpha<\omega_1\rangle$ with the following properties:
\begin{enumerate}
\item for each $A\in\mathcal{F}$, there is an $\alpha$ such that $\{x_\beta:\alpha\leq\beta<\omega_1\}\subseteq U$, and
\sk
\item if $\mathcal{U}$ is an open cover of $X$ in $V$, then there is an $\alpha$ such that any countable subset of $\{x_\beta:\alpha\leq\beta<\omega_1\}$
is covered by an element of $\mathcal{U}$.
\sk
\end{enumerate}

Let $\mathcal{U}$ be (in $V$) an open cover of $X$ by sets $U$ with $z\notin\cl_Y(U)$ (this is possible by regularity of $Y$). In $V[G]$,
we can find  $\alpha<\omega_1$ such that any countable subset of $\{x_\beta:\alpha\leq\beta<\omega_1\}$ is covered by an element of $\mathcal{U}$.
Now define
\begin{equation*}
W:=\{x_\beta:\alpha\leq\beta<\omega_1\}.
\end{equation*}

By (1) above, it follows that $z\in \cl_Y(W)$.  By our choice of $\alpha$, $z$ is not in the closure of any countable subset of $W$.  Thus, $W$
witnesses that $t(z, Y)>\aleph_0$, as needed.
\end{proof}

Now suppose {\sf CH} holds, let $\vec{X}=\langle X, Y, z\rangle$ be a bad triple, and let $G$ be a generic subset of $\mathbb{P}[\vec{X}]$.
In $V[G]$, we know there is a filter $\mathcal{F}$ as in the conclusion of Theorem~\ref{summary1}.  If $\vec{X}$ is still a bad triple in $V[G]$,
 then the pair $\langle X, \mathcal{F}\rangle$ satisfies the assumptions of Lemma~\ref{usefullemma}, and hence forcing with
  $\mathbb{P}[X,\mathcal{F}]$ over $V[G]$ will give us $t(z, Y)>\aleph_0$. In any case, we have demonstrated that a bad triple can be destroyed
  by a ``reasonable'' notion of forcing, as long as {\sf CH} holds.

The following two definitions give us some notation to keep things organized.

\begin{definition}
\label{type}
Let $\vec{X}$ be a relevant triple.
\begin{itemize}
\item We say $\vec{X}$ is of Type~0 if $\vec{X}$ is not a bad triple.
\sk
\item We say $\vec{X}$ is of Type~1 if it is a bad triple, but after forcing with $\mathbb{P}[\vec{X}]$ it is no longer bad.
\sk
\item We say $\vec{X}$ is of Type~2 if it is of neither of the previous two types.
\end{itemize}
\end{definition}

\begin{definition}
Let $\vec{X}=\langle X, Y, z\rangle$ be a relevant triple.
\begin{itemize}
\item We let $\mathbb{P}^0[\vec{X}]$ be the trivial (one-element) notion of forcing.
\sk
\item $\mathbb{P}^1[\vec{X}]$ denotes the forcing $\mathbb{P}[\vec{X}]$.
\sk
\item If $\vec{X}$ is of Type~2, we let $\mathbb{P}^2[\vec{X}]$ denotes the forcing $\mathbb{P}[\vec{X}]*\mathbb{P}[X,\dot{\mathcal{F}}]$, as discussed
in the paragraph prior to Definition~\ref{type}.
\sk
\end{itemize}
\end{definition}

We end with the following summary of our work in this section:

\begin{theorem}
\label{kill}
Suppose the Continuum Hypothesis holds. Then any relevant triple $\vec{X}$ is of Type~n for some unique $n\in\{0, 1, 2\}$.
Furthermore, if $\vec{X}$ is of Type~$n$, then
\begin{equation*}
V^{\mathbb{P}^n[\vec{X}]}\models\text{``$\vec{X}$ is not a bad triple''.}
\end{equation*}
\end{theorem}

\section{Putting it all together}

In this section we will build our model of {\sf ZFC + CH + $\circledast$}.  Our plan is  standard: we assume {\sf GCH} in the ground model
and build a countable support iteration of length $\omega_2$ where at each stage we treat a relevant triple using Theorem~\ref{kill}.

We start with an {\em ad hoc} definition:

\begin{definition}
A notion of forcing $\mathbb{P}$ is {\em admissible} if $\mathbb{P}$ is trivial, or of the form $\mathbb{P}[\vec{X}]$ for some bad triple $\vec{X}$,
or of the form $\mathbb{P}[X,\mathcal{F}]$ where $|X|=\aleph_1$ and $\langle X, \mathcal{F}\rangle$ is as in Definition~\ref{41}.
A countable support iteration $\bar{\mathbb{P}}=\langle \text{ $\mathbb{P}_\alpha$, $\dot{\mathbb{Q}}_\beta:\alpha\leq\omega_2$, $\beta<\omega_2$}\rangle$
is {\em admissible} if for each $\alpha<\omega_2$,
\begin{equation*}
\forces_{\alpha}\text{``$\dot Q_\alpha$ is admissible''}.
\end{equation*}
\end{definition}

Our goal is to show that there is an admissible iteration $\bar{\mathbb{P}}$ such that
\begin{equation}
\forces_{\omega_2}\text{``{\sf CH } + $\circledast$''}.
\end{equation}

Most of the results in this section will consist of pointers back to results already in the literature. The conclusions we draw are simply stated,
but the results we need to obtain these conclusions often involve very technical concepts. Our plan is to treat much of the earlier
work as a black box, as we do in the next proposition.

\begin{proposition}
Admissible iterations satisfy the assumptions of Theorem~4 in~\cite{EI15}. Thus, the limit $\mathbb{P}_{\omega_2}$ is totally proper, and
for any $\alpha<\omega_2$ the quotient forcing $\mathbb{P}_{\omega_2}/\dot{G}_\alpha$ is totally proper.
\end{proposition}
\begin{proof}
Let $\bar{\mathbb{P}}$ be an admissible iteration. Theorem~4 of~\cite{EI15} requires the iterands of $\bar{\mathbb{P}}$ to have three
properties:  they each must be totally proper, weakly $<\omega_1$-proper, and satisfy a complicated ``iteration condition''. We do not need the details of
these definitions here because sections 4, 5, and 6 of~\cite{EI17} are devoted to showing that forcings of the form $\mathbb{P}[X,\mathcal{F}]$ satisfy these three conditions, and
countably closed forcings satisfy them in a trivial way. The conclusion of Theorem~4 of~\cite{EI15} is that $\mathbb{P}_{\omega_2}$ is totally proper,
and the result regarding quotient forcing follows from Proposition~6.13 of the same paper.
\end{proof}

The quotient forcing portion of the preceding proposition is actually saying something quite simple: if $G$ is a generic subset of $\mathbb{P}_{\omega_2}$
and $G_\alpha$ is the canonical generic subset of $\mathbb{P}_\alpha$ obtained from $G$, then $V[G]$ is obtained from $V[G_\alpha]$ by forcing with a
totally proper notion of forcing.

\begin{corollary}
\label{cor1}
Let $\bar{\mathbb{P}}$ be an admissible iteration, and let $G$ be a generic in subset of $\mathbb{P}_{\omega_2}$.
\begin{enumerate}
\item If {\sf CH} holds in $V$, it continues to hold in $V[G]$.
\sk
\item Suppose $\alpha<\omega_2$ and $\vec{X}$ is a relevant triple in $V[G_\alpha]$. Then $\vec{X}$ is a relevant triple in $V[G]$.  Furthermore, if
$\vec{X}$ is not a bad triple in $V[G_\alpha]$, then it is not a bad triple in $V[G]$.
\end{enumerate}
\end{corollary}
\begin{proof}
The first statement follows since $\mathbb{P}_{\omega_2}$ is totally proper.  For the second, let $\vec{X}=\langle X, Y, z\rangle$ be a
relevant triple in $V[G_\alpha]$. Since $V[G]$ is obtained from $V[G_\alpha]$ through forcing with a totally proper notion of forcing,
we know  $X$ has the same countable subsets in both $V[G_\alpha]$ and $V[G]$, and hence $X$ remains countably compact in $V[G]$.  The other aspects
of relevance are preserved automatically.

Suppose now that $\vec{X}$ is not a bad triple in $V[G_\alpha]$.  This can happen by one of two ways: either $h\pi\chi(X)>\aleph_0$, or $t(z, Y)>\aleph_0$.
Destroying either of these conditions would require adding a new countable sequence of elements of $V[G_\alpha]$ --- either a countable sequence
of basic open sets of $X$ in the first case, or a new  countable of points of $X$ in the second. This cannot happen as $V[G]$ is obtained from $V[G_\alpha]$
via a totally proper notion of forcing.
\end{proof}

The above corollary tells us that if a bad triple is destroyed at some stage of an admissible iteration, then it stays dead. Clearly this will be an
important ingredient in our construction, but we need to be able to argue that every relevant triple in the final extension has had its ``badness'' destroyed
at some stage along the way.  To do this, we need to show that the limit forcing $\mathbb{P}_{\omega_2}$ has the $\aleph_2$-chain condition.

Again, this will follow after some appeals to the literature.  In this case, we will need to take advantage of previous work on the so-called
{\em $\aleph_2$-properness isomorphism condition} (abbreviated $\aleph_2$-p.i.c.).
We will not give the (complicated) definition here, as all we need to know is encapsulated in the following three propositions:

\begin{proposition}
\label{blackbox1}
 Suppose $\bar{P}=\langle \mathbb{P}_\alpha,\dot{\mathbb{Q}}_\beta:\alpha\leq\omega_2, \beta<\omega_2\rangle$ is a countable support iteration
such that
\begin{equation*}
\forces_\alpha\text{``$\dot{\mathbb{Q}}_\alpha$ satisfies the $\aleph_2$-p.i.c.''}.
\end{equation*}
Then
\begin{enumerate}
\item  $\mathbb{P}_\alpha$ satisfies the $\aleph_2$-p.i.c for each $\alpha<\omega_2$, and
\sk
\item if $2^{\aleph_0}=\aleph_1$ then $\mathbb{P}_{\omega_2}$ satisfies the $\aleph_2$-c.c.
\end{enumerate}
\end{proposition}
\begin{proof}
This is a special case of Lemma~2.4 on page~410 of~\cite{pif}.  Abraham also gives a nice treatment of this in Section~5.4 of~\cite{proper}.
\end{proof}

\begin{proposition}
\label{blackbox3}
Let $\mathbb{P}$ be an admissible notion of forcing.  If {\sf CH} holds, then $\mathbb{P}$ satisfies the $\aleph_2$-p.i.c.
\end{proposition}
\begin{proof}
Lemma~2.5 on page~411 of~\cite{pif} tells us that a proper notion of forcing of size $\aleph_1$ satisfies the $\aleph_2$-p.i.c.
Clearly nothing more needs to be said for the trivial forcing, and under {\sf CH} if $\vec{X}$ is a bad triple, then  $\mathbb{P}[\vec{X}]$ is of cardinality~$\aleph_1$
and hence satisfies the $\aleph_2$-p.i.c. as well. For $\mathbb{P}$  of the form $\mathbb{P}[X,\mathcal{F}]$ with $|X|=\aleph_1$, we get the $\aleph_2$-p.i.c. by way of Theorem~7.2 of~\cite{EI17}.
\end{proof}

Now we can put the pieces together to obtain the following lemma:

\begin{lemma}
Suppose $\bar{\mathbb{P}}$ is an admissible iteration, and $2^{\aleph_0}=\aleph_1$ and $2^{\aleph_1}=\aleph_2$ in the ground model. Then
\begin{itemize}
\item $\mathbb{P}_{\omega_2}$ satisfies the $\aleph_2$-chain condition, and
\sk
\item $\forces_{\omega_2}\text{``}2^{\aleph_1}=\aleph_2\text{''}$.
\end{itemize}
\end{lemma}
\begin{proof}
We know that our iteration satisfies the assumptions of Proposition~\ref{blackbox1}, so for each $\alpha<\omega_2$, $\mathbb{P}_\alpha$ satisfies the
 $\aleph_2$-p.i.c., and since {\sf CH} holds we know $\mathbb{P}_{\omega_2}$ has the $\aleph_2$-chain condition.
We preserve $2^{\aleph_1}=\aleph_2$ for standard reasons, as we are iterating $\aleph_2$-c.c. posets, and (by induction) each iterand is forced
to have cardinality at most $2^{\aleph_1}=\aleph_2$.
\end{proof}

We now have everything we need to construct our model.

\begin{theorem}
Assume $2^{\aleph_0}=\aleph_1$ and $2^{\aleph_1}=\aleph_2$.  There is an admissible iteration $\bar{\mathbb{P}}$ so that
\begin{equation}
\forces_{\omega_2}\text{`` {\sf CH} } + \circledast \text{''}.
\end{equation}
\end{theorem}
\begin{proof}
The iteration is built by induction on $\alpha<\omega_2$.  At stage $\alpha$, some fixed bookkeeping procedure will hand us a $\mathbb{P}_\alpha$~name $\dot{\vec{X}}_\alpha$ for a relevant triple.
Since {\sf CH} holds in $V^{\mathbb{P}_\alpha}$, we know
\begin{equation*}
\forces_\alpha\text{``$\dot{\vec{X}}_\alpha$ is of Type $n$ for some $n<3$''},
\end{equation*}
and we let $\dot{\mathbb{Q}}_\alpha$ be a $\mathbb{P}_\alpha$ name such that
\begin{equation*}
\forces_\alpha\text{``$\dot{\mathbb{Q}}_\alpha$ is $\mathbb{P}^n[\dot{\vec{X}}_\alpha]$ where $\dot{\vec{X}}_\alpha$ is of Type $n$.''}
\end{equation*}
Standard arguments using the $\aleph_2$-chain condition allow us to arrange the bookkeeping so that every relevant triple in the
final extension $V^{\mathbb{P}_{\omega_2}}$ is considered at some stage of the iteration.  This is where we use our assumption that relevant
triples consist of topological spaces with underlying set $\omega_1$, as this guarantees there are only $\aleph_2$ relevant triples that need to be considered.

Now let $ {G}$ be a generic subset of $\mathbb{P}_{\omega_2}$,
and for $\alpha<\omega_2$ let $G_\alpha$ be the generic subset of $\mathbb{P}_\alpha$
obtained from $ {G}$.  We know that {\sf CH} is true in $V[G]$ by virtue of Corollary~\ref{cor1}, so we need only worry about $\circledast$.

Since {\sf CH} holds in $V[G]$, Proposition~\ref{prop2} tells us that it suffices to show that there are no bad triples in $V[G]$. If $\vec{X}$
is a relevant triple in $V[G]$, there is an $\alpha<\omega_2$ such that  $\vec{X}\in V[G_\alpha]$, and $\vec{X}=\vec{X}_\alpha$.

Given the forcing we do at stage $\alpha$, we know that $\vec{X}$ is not a bad triple in $V[G_{\alpha+1}]$, and we can then apply the second part of Corollary~\ref{cor1}
to conclude that $\vec{X}$ is not a bad triple in $V[G]$.  Therefore, $\circledast$ holds in $V[G]$.
\end{proof}

\section{Questions}

Finally, we collect here a few questions.  A few of these are well-known and taken from Shakhmatov's~\cite{Shak1992}, while others are specifically motivated by the construction presented here.

\begin{question}
Is it consistent that every countably compact, countably tight space is sequential?
\end{question}

This is Problem~2.12 from~\cite{Shak1992}. Dow~\cite{dow2} has shown that a counterexample exists if $\mathfrak{b}=2^{\aleph_0}$.

The proof of Proposition~\ref{chmm} required both $\circledast$ and $2^{\aleph_0}<2^{\aleph_1}$ in order to establish that compact spaces of countable tightness are sequential.
It is not clear if $\circledast$ alone is sufficient for this:

\begin{question}
Does $\circledast$  imply that compact spaces of countable tightness are sequential?
\end{question}

The preceding question is of course closely related to Ismail and Nyikos's Problem (\cite{ismailnyikos}, but see also Problem~3.2 in~\cite{Shak1992}).
\begin{question}
If a compact space is $C$-closed, must it be sequential?
\end{question}

Our proof of the consistency of $\circledast$ made heavy use of the fact that the Continuum Hypothesis held in the final model.  It is not clear
how to obtain the result without this assumption, and this leads us to the following question:

\begin{question}
 Is $\circledast$ consistent with the failure of {\sf CH}?  Does it follow from {\sf PFA}?
 \end{question}

Finally, we finish with a related question mentioned in~\cite{dow2}:

\begin{question}
Is it consistent that every regular countably tight space is $C$-closed?
\end{question}

\bibliographystyle{plain}

\end{document}